\documentclass[11pt,a4paper, envcountsame]{amsart}
\usepackage{latexsym,fullpage}
\usepackage[usenames,dvipsnames]{color}

 \usepackage{amsmath, amssymb, xspace,url}
 \usepackage{graphicx}

 \usepackage[all]{xy}

\usepackage{tikz}
\usetikzlibrary{arrows,snakes,positioning,backgrounds,shadows}
 
\newtheorem{theorem}{Theorem}[section]

\newtheorem{thm}[theorem]{Theorem}

\newtheorem{remark}[theorem]{Remark}

\newtheorem{prop}[theorem]{Proposition}

\newtheorem{claim}[theorem]{Claim}

\newtheorem{conjecture}[theorem]{Conjecture}

\newtheorem{lemma}[theorem]{Lemma}

\newtheorem{cor}[theorem]{Corollary}

\theoremstyle{definition}
\newtheorem{definition}[theorem]{Definition}

\newcommand{\NN}{{\mathbb{N}}}
\newcommand{\RR}{{\mathbb{R}}}

\newcommand{\sub}{\subseteq}
\newcommand{\sN}[1]{_{#1\in \NN}}
\newcommand{\uhr}[1]{\! \upharpoonright_{#1}}
\newcommand{\ML}{Martin-L{\"o}f}
\newcommand{\SI}[1]{\Sigma^0_{#1}}
\newcommand{\PI}[1]{\Pi^0_{#1}}
\newcommand{\PPI}{\PI{1}}

\newcommand{\bi}{\begin{itemize}}
\newcommand{\ei}{\end{itemize}}
\newcommand{\bc}{\begin{center}}
\newcommand{\ec}{\end{center}}

\newcommand{\ES}{\emptyset}

\newcommand{\tp}[1]{2^{#1}}
\newcommand{\ex}{\exists}
\newcommand{\fa}{\forall}

\newcommand{\Kuc}{Ku{\v c}era}
\newcommand{\seqcantor}{2^{ \NN}}

\newcommand{\cantor}{\seqcantor}
\newcommand{\strcantor}{2^{ < \omega}}

\newcommand{\Opcl}[1]{[#1]^\prec}

\newcommand{\n}{\noindent}

\newcommand{\leb}{\mathbf{\lambda}}

\newcommand{\sss}{\sigma}

\newcommand \seq[1]{{\left\langle{#1}\right\rangle}}

\newcommand\+[1]{\mathcal{#1}}

\newcommand{\wt}{\widetilde}

\newcommand{\lra}{\leftrightarrow}
\newcommand{\LR}{\Leftrightarrow}
\newcommand{\RA}{\Rightarrow}
\newcommand{\LA}{\Leftarrow}

\newcommand{\sssl}{\ensuremath{|\sigma|}}

\title{Multiple recurrence and algorithmic randomness}

\author{Rodney G. Downey}
\address{Department of Mathematics, Statistics and Operations
  Research, Victoria University of Wellington, Wellington,  New Zealand}
\email{rod.downey@msor.vuw.ac.nz}

\author{Satyadev Nandakumar}
\address{Department of Computer Science and Engineering, Indian
  Institute of Technology Kanpur, Kanpur, Uttar Pradesh, India.}
\email{satyadev@cse.iitk.ac.in}

\author{Andr\'{e} Nies}
\address{Department of Computer Science, University of Auckland, Auckland, New
  Zealand} 
\email{andre@cs.auckland.ac.nz}


\begin{document}

\maketitle

\begin{abstract}  This work contributes to the programme of studying  effective versions of ``almost everywhere" theorems in analysis and ergodic theory via algorithmic randomness. We determine the level of randomness needed for a point in  Cantor space $ \{0,1\}^{\NN}$ with the uniform measure and the usual shift so that    effective versions of the  multiple recurrence theorem of Furstenberg holds for iterations starting at the point.  We consider recurrence into  closed sets that possess various degrees of effectiveness: clopen, $\PPI$ with computable measure, and $\PPI$.  The notions of Kurtz, Schnorr, and \ML\ randomness, respectively, turn out to be sufficient. We obtain similar   results for multiple recurrence with respect to  the $k$ commuting shift operators  on $\{0,1\}^{\NN^{\normalsize k}}$. \end{abstract}

  \section{Introduction} 
A major subarea of mathematical logic seeks to understand the \emph{effective 
content} of mathematics; that is, to understand what part of mathematics 
is algorithmic and to calibrate the computational resources 
needed for classical theorems. 
Ever since Turing's original paper \cite{Turing36}  {analysis}
has been part of this tradition. The last decade or so has seen 
a great deal of work using algorithmic tools to understand and 
calibrate our intuitive understanding of randomness of an individual 
sequence. (Downey and Hirschfeldt \cite{Downey.Hirschfeldt:book}, Nies \cite{Nies:book}, Li-Vitanyi
\cite{LV}). For example, a real can be regarded as random if 
no effective betting strategy succeeds in making infinite 
capital betting on the bits of the real. A natural area for the combination of 
these two parts of effective mathematics is in the area of 
``almost everywhere'' mathematics. 
For example, Brattka, Miller and Nies \cite{BMN} have established that 
various  randomness properties of a real correspond precisely to 
differentiability at the real of, for instance, computable functions of bounded variation, or computable nondecreasing functions.

This paper contributes to this program. 
We consider   one of the most powerful areas of mathematics 
arising in the 20th century, beginning with  the work of Poincair{\'e}, Birkhoff,
von Neumann and others: \emph{ergodic theory}.
This theory is concerned with ``average case long term behavior''
of certain kind of recurrent systems, and has applications from 
pure mathematics to classical physics. 

In this paper we make the first contributions to studying a celebrated results in ergodic theory,  Furstenberg's
multiple recurrence  theorem (see the new edition of Furstenberg's book~\cite{Furstenberg:2014}). This theorem showed that 
methods from ergodic theory could be used to 
derive important results in additive number theory, and 
has ushered in a revolution in this area.

\section{Background in ergodic theory}

We begin with some background that will lead to our main definition.
  A measurable  operator $T\colon  X \to X$   on a probability space  $(X, \+ B, \mu)$  is called \emph{measure preserving} if   $\mu T^{-1}(A) = \mu A$ for each $ A \in \+ B$. 
We say that 
$A\in \+ B$ is \emph{invariant} under $T$ if $T^{-1}A=A$ (up to a null set).
Finally,
$T$ is \emph{ergodic} if 
the only $T$-invariant measurable subsets of $X$ are either 
null or co-null\footnote{$T$ being ergodic is implied by 
various
\emph{mixing} properties, for us  relevent one is 
\emph{strong mixing} which means that  
for $A,B\in \+ B$,
$\lim_{n\to \infty} \mu(A\cap T^{-n} B)=\mu(A)\mu(B).$}.

A classical result here is that almost all points in a probability space behave 
in a regular way.
For example,
consider the following, essentially the first ergodic theorem.

\begin{theorem}[Poincair{\'e} Recurrence Theorem]
Let $(X,\+ B,\mu)$ be a probability space,   $T\colon X\to X$ be 
measure preserving, and let $A\in \+ B$ have  positive measure. 
For almost all $x\in X$, 
$T^{-n}(x)\in A$ for infinitely many $n$.
\end{theorem}

The remarkable theorem of Furstenberg says that 
in certain circumstances, $T^n(x)\in E$ for a  collection of $n$ that forms an 
arithmetical progression.
More generally,
suppose we have $k$ commuting measure preserving operators.  There is an $n$ and  a positive measure set of points so that  $n$  iterations   of each of the operators  $T_i$,  starting from each   point in the set, ends in $A$:
\begin{theorem}[Furstenberg  multiple recurrence theorem, see~\cite{Furstenberg:2014} Thm.\ 7.15]  \label{thm:FurMRT2} \ \\ Let $(X, \+ B, \mu)$ be a probability space. Let $T_1, \ldots, T_k$ be commuting measure preserving operators on $X$. 
Let    $A \in \+B$ with $\mu A > 0$.  There is 
$n>0$ such that $0< \mu \bigcap_i T_i^{-n}(A)$.   \end{theorem}

For this paper the following equivalent formulation  of Thm.\  \ref{thm:FurMRT2} will matter. We verify their equivalence in Subsection~\ref{ss:equivalence}.
\begin{cor} \label{prop:FurMRT3} With the hypotheses of Thm.\ \ref{thm:FurMRT2}, for $\mu$-a.e.\ $x \in A$, there is an $n>0$ such that $x \in \bigcap_i T_i^{-n}(A)$. \end{cor}

%

An important   special case is that   $T_i $ is the power $ V^i$ of  a measure-preserving  operator $V$. 
\begin{cor} \label{cor:MRT} Let $(X, \+ B, \mu)$ be a probability space. Let $V$ be  a  measure preserving operator.  Let $A \in \+ B$  and $\mu A > 0$. For each $k$, for $\mu$-a.e.\ $x \in A$   there is $n$ such that $\fa i. {1 \le i \le k}  \, [x \in V^{-ni}( A)]$. 
\end{cor}
This is the form which, as  Furstenberg \cite{Furstenberg:2014} showed,  can be used to derive 
van der Waerden's theorem on arithmetic progressions; see e.g. Graham, Rothschild and Spencer~\cite{GRS}.
(A full      ``almost everywhere" version of Corollary \ref{cor:MRT} would assert that for $\mu$-a.e.\ $x$ there is $n$ such that  $\fa i. {1 \le i \le k}  \, [  x \in V^{-ni}( A)]$. 
Note that we can expect such a version   only    for ergodic operators, even if $k=1$. For, if $A$ is $S$-invariant,   then an iteration starting from $x \not \in A$ will never get into~$A$.)

We will mostly assume that   $(X, \+ B, \mu)$   is  Cantor space $\cantor$ with the product measure $\leb$.  In the following $X,Y,Z$ will denote elements of Cantor space. We will work with   the shift  $T \colon \cantor \to \cantor$  as the measure preserving operator. Thus, $T(Z)$ is obtained by deleting the first entry of the bit  sequence $Z$. We note that this operator    is (strongly)  mixing, and hence strongly ergodic, namely, all of its powers are ergodic.

The following is our central definition.\begin{definition} \label{def:central} Let $\+ P \sub \cantor$ be measurable, and let  $Z \in \cantor$. We say that   $Z$ is  \emph{$k$-recurrent  in $\+ P$}  if  there is $n \ge 1$ such that 
\[ \tag{$\diamond$} Z \in \bigcap_{1 \le i \le k} T^{-ni} (\+ P). \]
We  say that $Z$ is \emph{multiply recurrent} in $\+ P$ if $Z$ is $k$-recurrent in $\+ P$  for each $k \ge 1$.
\end{definition}
In other words, $Z$ is $k$-recurrent in $\+ P$ if there is $n \ge 1$ such that removing  $n, 2n, \ldots, kn$   bits from the beginning of   $Z$ takes us into $\+ P$. 

We only consider multiple recurrence for  closed sets.   Note that for (multiple) recurrence in the sense of Thm.\ \ref{thm:FurMRT2},   this  is not an essential restriction, because any set of positive measure contains a closed subset of positive measure.

\subsection*{Goal of  this paper} 
Algorithmic information theory is an area of 
research which attempts to give meaning to 
randomness of individual events.  
We remind the reader that $A\in 2^\omega$ is 
Martin-L\"of random if $A\not\in \cap_{i} U_i$
where $\{U_i\colon i\in {\mathbb N}\}$ is a
computable collection of $\Sigma_1^0$ classes with 
$\mu(U_i)\le 2^{-i}$, and $B$ is \emph{Kurtz} random 
if $A\in Q$ for every $\Sigma_1^0$ class of measure 1\footnote{For
 background on randomness
 notions see \cite[Ch.\ 3]{Nies:book} or \cite[Chapters 6 and
   7]{Downey.Hirschfeldt:book}.} 

 Since limit laws in probability theory embody certain intuitive
properties associated with randomness, it is important to establish
that objects of study in algorithmic randomness satisfy such
laws. Indeed, the Law of Large Numbers \cite{vanLambalgen}, the Law of
Iterated Logarithm \cite{Vovk87},  Birkhoff's Ergodic theorem
\cite{V97,FGMN12} and the Shannon-McMillan-Breiman Theorem
\cite{Hochman09,Hoyrup12} are satisfied by Martin-L\"of random
sequences. These results serve to justify the intuition that the set
of Martin-L\"of random points is a canonical set of measure 1 which obeys all
effective limit laws that  hold almost everywhere.
 
 In this vein, we   analyse how weaker and weaker
 effectiveness conditions on  a closed set $\+ P$ ensure  multiple recurrence when
 starting from a  sequence $Z$ that satisfies a stronger and stronger
 randomness property for an algorithmic test notion. We begin with the
 strongest effectiveness condition, being clopen; in this case it is
 easily seen that   weak (or Kurtz)   randomness of $Z$  suffices. As
 most general effectiveness condition  we will  consider being
 effectively closed (i.e.\ $\PI 1$); \ML-randomness turns out to be
 the appropriate notion. The proof will import some method from the
 case of a clopen $\+ P$.  Note that $\PI 1$ subsets of Cantor space
 are often called \emph{$\PI 1$ classes}.

The major part of the present work establishes the multiple recurrence
property of algorithmically random sequences under the left-shift
transformation. In the question of ``structure versus randomness'' in
measure-preserving dynamical systems, this represents a setting which
has a high degree of randomness. In a final section, we also provide a
proof of the principle, well-known in practice, that a very structured
system, namely Kronecker systems also exhibit multiple recurrence for
all points.

\subsection*{Related results} In the theory of algorithmic randomness,
we are interested in which limit laws in probability theory are 
applicable to algorithmically random objects. Of particular 
interest are computable versions of theorems in dynamical
systems. An early result is the   theorem of Ku{\v c}era~\cite{Kucera:85} that a
sequence $X$ is Martin-L\"of random if and only if for every $\PPI$
class $P$ with positive measure, there is a tail of $X$ which is in
$P$. This result can be recast into ergodic-theoretic
language using the left-shift transformation. Bienvenu, Day, Hoyrup,
Mezhirov and Shen \cite{BDHMS12} have generalized this to arbitrary
computable ergodic transformations. They show that for a computable
probability space and a computable ergodic transformation $T$, a
member $x$ in the space is Martin-L\"of random if and only if for
every $\PPI$ set $P$ of positive measure, there is an $n$ such that
$T^n x \in P$. These can be viewed as how Poincar\'e recurrence
relates to algorithmic randomness notions.

Birkhoff's ergodic theorem states that for every ergodic transformation $T$
defined on a probability space $(X,\mathcal{F}, \mu)$ and for every
function $f \in L^{1}(X)$, the average $\frac 1 n 
\sum_{i=0}^{n-1} f(T^ix)$ converges to $\int f d\mu$ for
$\mu$-almost every $x \in X$. The question of how Birkhoff's ergodic
theorem relates to algorithmic randomness has also been
studied, starting with V'yugin~\cite{V97}. G\'{a}cs, Hoyrup and Rojas \cite{Gacs.Hoyrup:11} establish
that a point $x$ in a computable probability space under a computable
ergodic transformation $T: X \to X$ is Schnorr random\footnote{Recall
that $C$ is called \emph{Schnorr} random if 
$C\not\in \cap_i U_i$ 
where $\{U_i\colon i\in {\mathbb N}\}$ is a
computable collection of $\Sigma_1^0$ classes with
$\mu(U_i)= 2^{-i}$; like Kurtz randomness, a notion strictly weaker 
than Martin-L\"of randomness.}
   if and only if
it obeys the Birkhoff ergodic theorem with respect to all $\PPI$ sets
of computable positive measure --- \emph{i.e.} the asymptotic
frequency with which the orbit of $x$ visits such a $\PPI$ set $P$ is
exactly the probability of~$P$.  Franklin, Greenberg, Miller and Ng
\cite{FGMN12} show that under the same setting, but without the
assumption that the measure of the $\PPI$ set be  computable, a point
$x$ is Martin-L\"of random if and only if it obeys the Birkhoff
ergodic theorem with respect to $\PPI$ sets of positive measure.


\subsection*{Proof that Thm.\ \ref{thm:FurMRT2} and   Cor.\  \ref{prop:FurMRT3} are equivalent. }   \label{ss:equivalence} 
%
 Cor.\  \ref{prop:FurMRT3}   clearly yields Thm.\ \ref{thm:FurMRT2} because it implies that $  \bigcap_i T_i^{-n}(A)$   has positive measure for some $n$.  Conversely,  let us  show that  Thm.\ \ref{thm:FurMRT2} yields   Cor.\  \ref{prop:FurMRT3}. 
Let  $R_n = \bigcap_i T_i^{-n}(A)$.  We recursively define a sequence $\seq {n_p}_{p< N}$ of numbers and a descending sequence $\seq {A_p}_{p< N}$ of sets, where $0<N \le \omega$.

Let $n_0 = 0$, and  $A_0 = A$. Suppose  $n_p$ and $A_p$ have been defined. If  $\mu A_p =0$ let $N=p+1$ and finish. Otherwise, let $n_{p+1}$ be the least $n > n_p$ such that $\mu  (A_p \cap R_n) >0$, and   let  $A_{p+1} = A_p - R_n$. 

Let $A_N= \bigcap_{p<N} A_p$. Then $\mu A_N =0$: This is clear if $N$ is finite. If $N= \omega$ and $\mu A_N >0$, by Thm.\ \ref{thm:FurMRT2}  there  is $n$ such that 
$\mu( R_n \cap A_N) > 0$. This contradicts the definition of $A_{p+1}$ where $n_p < n \le n_{p+1}$.

Since $\mu A_N =0$,  Cor.\  \ref{prop:FurMRT3}  follows.

\subsection*{Notation}
For a set of strings $S \sub \strcantor$, by $\Opcl S$ we denote the open set $\{ Y  \in \cantor \colon \, \ex \sss \in S \, [\sss \prec Y]\}$. 
We write $\leb \Opcl S$ for the measure of this set, namely $\leb (\Opcl S)$.

Recall that we  work with   the shift  $T \colon \cantor \to \cantor$  as the measure preserving operator.  We will   write   $Z_n$ for $T^n (Z)$,   the tail of $Z$ starting at bit position $n$.  Thus, for any $\+ C \sub \cantor$,  $Z \in \cantor$ and $k \in \NN$, $Z \in T^{-k} (\+ C) \lra Z_k \in \+ C$.

\section{Multiple recurrence for weakly  random sequences} 
Recall   that $Z$ is weakly  (or Kurtz) random if $Z$ is in no null $\PI
1$ class.  This formulation is equivalent to the usual one in terms of effective tests.
\begin{prop} \label{prop: Kurtz} Let $\+ P \sub \cantor$  be a
  non-empty  clopen set. Each weakly   random bit sequence $Z$ is
  multiply recurrent  in $\+ P$.\end{prop}

\begin{proof}

Suppose    $Z$ is not $k$-recurrent in $\+ P$ for some $k \ge 1$. We define a null $\PPI$ class  $\+ Q$ containing $Z$. Let $n_0$ be least such that $\+ P = \Opcl F$ for some set of strings of length $n_0$.  Let $n_t= n_0(k+1)^t$ for $t \ge 1$. Let
\[ \+ Q = \bigcap_{t \in \NN}  \{ Y \colon   \bigvee_{1 \le i \le k} Y_{i n_t} \not \in \+ P\}. \]
 By definition  of $n_0$  the conditions in  the same disjunction are independent, so  we have \[ \leb (\bigvee_{1 \le i \le k} Y_{i n_t} \not \in \+ P ) =  1 - (\leb \+ P) ^k < 1.\] By definition of the $n_t$ for $t >0$, the class $\+ Q$ is the independent intersection of  such classes  indexed by~$t$. Therefore $\+ Q$ is null. Clearly $\+ Q$ is $\PPI$.
 
   By hypothesis $Z \in \+ Q$. So $Z$ is not weakly random.
\end{proof} 

\section{Multiple recurrence for     Schnorr random sequences} 
 
\begin{thm} Let $\+ P \sub \cantor$  be a $\PPI$ class such that   $0< p = \leb  \+ P$ and $p$ is a computable real. Each Schnorr   random   $Z$ is multiply recurrent  in $\+ P$.
\end{thm}
We note that   this  also follows from a particular kind of  effective version of  Furstenberg multiple recurrence (Cor.\ \ref{prop:FurMRT3}), as  explained in Remark~\ref{Rem:Rute} below. However, we prefer to give a direct proof avoiding Cor.\ \ref{prop:FurMRT3}.
\begin{proof} We extend the previous proof, working with an effective approximation $\+ B= \cantor - \+ P = \bigcup_s  \+ B_s$ where the $\+ B_s$ are clopen. We may assume that $ \+ B_s = \Opcl {B_s}$ for some effectively given set $B_s$ of strings of length $s$. 

We fix an arbitrary $k \ge 1$ and show that $Z$ is    $k$-recurrent in $\+ P$. 
 Given $v \in \NN$ we will define a null $\PPI$ class $\+ Q_v \sub \cantor$ which plays a role similar to the class $\+ Q$  before. We also define an  ``error class'' $\+ G_v \sub \cantor$ that is $\SI 1$ uniformly in $v$. Further, $\leb \+ G_v$ is computable uniformly in $v$ and $\leb \+ G_v \le \tp {-v}$, so that $\seq { \+ G_v }\sN v$ is a Schnorr test.  If $Z$ passes this Schnorr test then $Z$ behaves essentially like a weakly random in the  proof of  Proposition~\ref{prop: Kurtz}, which shows that $Z$ is $k$-recurrent for  $\+ P$. 
 
For the details, given  $v \in \NN$, we define a computable sequence $\seq {n_t}$.  Let $n_0= 1$.
   Let $n= n_t \ge (k+1)n_{t-1}$ be so large that \bc $\leb (\+ B - \+ B_n) \le \tp{-t- v-k}$. \ec
   
As in the proof of Proposition~\ref{prop: Kurtz}, the  class
\[ \+ Q_v = \{ Y \colon \fa t  \bigvee_{1 \le i \le k} Y_{i n_t} \in \+ B_{n_t}\} \]
is  $\PPI$ and null. The ``error class" for  $v$ at stage $t$ is
\[ \+ G_{v}^t = \{ Y \colon \bigvee_{1 \le i \le k} Y_{i n_t} \in  \+ B - \+ B_{n_t}\}. \]
Notice  that $\leb \+ G^t_v \le k \tp{-t- v-k}$, and this measure is computable uniformly in $v,t$. Let $\+ G_v = \bigcup_t G^t_v$. Then $\leb \+ G_v$ is also uniformly computable in $v$, and bounded above by $\tp {-v}$, as required. 

If  $Z$ is Schnorr random, there is $v$ such that $Z \not \in \+ G_v$. Also, $Z \not \in \+ Q_v$, so that for some $t$ we have $Z_{i n_t} \in \+ P$ for each $i$ with  $1 \le i \le k$, as required.
\end{proof}

\section{Multiple recurrence for ML-random sequences}
For  general  $\PPI$ classes, the right level of randomness to obtain multiple recurrence is \ML\ randomness. We first remind the reader that  even the case of $1 $-recurrence  characterizes ML-randomness. This is a well-known result of \Kuc\  \cite{Kucera:85}.

\begin{prop} $Z$ is ML-random $\LR$ 
	 $Z$ is $1$-recurrent in  each  $\PPI$ class $\+ P$ with  $ 0<  p= \leb  \+ P$.  \end{prop}

\begin{proof} $\RA$:  see e.g.\ \cite[3.2.24]{Nies:book} or
  \cite[6.10]{Downey.Hirschfeldt:book}.
	
	\n $\LA$: ML-randomness of a sequence $Z$  is preserved by   adding bits at the beginning. By the Levin-Schnorr Theorem, the $\PPI$ class $\+ P = \{ Y \colon \fa n  K(Y \uhr n) \ge n-1\}$ consists entirely of ML-randoms. So, if $Z$ is not ML-random, then no tail of $Z$ is in the $\PPI$ class $\+ P$. Further, $\leb \+ P \ge 1/2$.
\end{proof}

\begin{thm} \label{thm:ML} Let $\+ P \sub \cantor$  be a $\PPI$ class with  $ 0<  p= \leb  \+ P$. Each \ML\    random   $Z$ is multiply recurrent  in $\+ P$.
\end{thm}

\begin{proof}
As before we fix an arbitrary $k \ge 1$ in order to  show that $Z$ is    $k$-recurrent in $\+ P$.  First we prove the assertion under the additional assumption  that $1-1/k <  p$.  This  generalises \Kuc's argument in  `$\RA$' of  the proposition above,  where $k=1$ and the additional assumption $0<p$ is already satisfied.

Let $B \sub \strcantor$ be a prefix-free c.e.\ set such that $\Opcl B = \cantor - \+ P$. We may assume that $B_0 = \ES$ and for each $t>0$, if $\sss \in B_t - B_{t-1}$ then $\sssl = t$. We define a uniformly c.e.\ sequence $\seq{C^r}$ of prefix-free sets which also have  the   property that at stage $t$ only strings of length $t$ are enumerated.
 
 For a string $\eta$ and $u \le |\eta|$, we write $(\eta)_u$ for  the string $\eta$ with the first $u$ bits removed.
 Let  $C^0$ only contain the empty string,  which is enumerated at stage $0$. Suppose  $r > 0$ and $C^{r-1}$  has been defined. Suppose $\sss$ is enumerated in $C^{r-1}$ at stage $s$ (so $\sssl =s$). For strings $\eta \succ \sss$ we search for the failure  of    $k$-recurrence in $\+ P$ that would be obtained by  taking  $s$ bits off $\eta$  for $k$ times.
At stage $t>   (k+1) s$,  for each string  $\eta $ of   length $t$  such that   $\eta \succ \sss$ and 
\[ \tag{$*$} \bigvee_{1 \le i \le k}  (\eta)_{si} \in B_{t-si}, \]
and no prefix of $\eta$ is in $C^r_{t-1}$,   put $\eta$  into $C^r$ at stage $t$.   
\begin{claim} $C^r$ is prefix-free for each $r$. \end{claim}
This holds for $r=0$. For $r> 0$ suppose that $\eta \preceq \eta'$ and both strings are in $C^r$. Let $t= |\eta|$. By inductive hypothesis the string $\eta$ was enumerated into $C^r$ via a unique  $\sss \prec \eta$, where $\sss \in C^{r-1}$. Then $\eta = \eta'$ because we chose the string in $C^{r}$  minimal under the prefix relation. This establishes the claim.
 
 By hypothesis $1> q = k \leb \Opcl B$. 
\begin{claim} \label{cl:small measure} For each $r \ge 0$ we have $\leb \Opcl{C^r}  \le q^r$. \end{claim}
This holds for $r=0$. Suppose  now that $r>0$. Let  $\sigma \in C^{r-1}$. The local measure above $\sss$ of strings $\eta $,  of a length $t$,  such that $\bigvee_{1 \le i \le k}  \eta_{si} \in B_{t-is}$ is at most $q$. The estimate follows by the prefix-freeness of $C^r$.

If   $Z$ is not $k$-recurrent  in $\+ P$, then $Z \in \Opcl{C^r}$ for each $r$, so $Z$ is not ML-random. 

We now remove the additional assumption that $1-1/k <  p$.  We define the sets $C^r$ as before.  Note that any string in $C^r$ has length at least $r$. 
  Everything will work except for Claim~\ref{cl:small measure}: if $\leb  \Opcl B\ge 1/k$ then $\leb \Opcl{C^r} $ could   be $1$. To remedy this,  we choose a finite set  $D \sub  B$ such that the set $\wt B = B- D$ satisfies $\leb \Opcl {\wt B} < 1/k$.  
Let $N = \max \{\sssl \colon \sss \in D\}$. We modify the argument of Prop.\ \ref{prop: Kurtz}, where  the clopen set  $\+ P$ there  now becomes  $\cantor - \Opcl D$.  

Let $C = \bigcup_r C^r$. Let $G_m$ be the set of prefix-minimal strings $\eta $  such that  $\eta  \in C$,  and there exist $m$ many $s> N$   as follows. \bi \item  $\eta \uhr s \in C$, and \item for some $i$ with  $1 \le i \le k$, $\eta \uhr{ [ si, s(i+1))}$ extends a string in $D$.  \ei 
(Informally speaking, if there are arbitrarily long such  strings along $Z$, then   the attempted test $\Opcl{C^r}$ might  not work, because the relevant ``block"  $\eta \uhr{ [ si, s(i+1))}$ may  extend a string  in $D$,  rather than one  in $\wt B$.)

The sets $G_m $ are uniformly $\SI 1$. By  choice of $N$ and independence, as in the proof of Prop.\ \ref{prop: Kurtz} we have $\leb \Opcl{G_{m+1}} \le (1- v^k) \leb G_m$, where $v = \leb (\cantor - \Opcl D)$. If  $Z$ is ML-random we can choose a least $m^*$   such that $Z \not \in \Opcl{G_{m^*}}$.  

Note that $m^*>0$ since $G_0 = \{\ES\}$. So choose $\rho \prec Z$ such that $\rho \in G_{m^*-1}$. Then 
 $\rho\in C^r$ for some  $r$, and  no $\tau$ with  $\rho \preceq \tau \prec Z$ is in $G_{m^*}$. 
 
 We define a ML-test that succeeds on  $Z$. Let $\wt C^r = C^r$. Suppose  $u >  r$ and $\wt C^{u-1}$ has been defined. For each  $\sss \in \wt C^{u-1}$,     put into   $\wt C^{u} $ all the  strings $\eta \succ \sss$ in $C^{u} $ so that    ($*$) can be strengthened to  $ \bigvee_{1 \le i \le k}  (\eta)_{si} \in \wt B_{t-is}$, where $s = \sssl$. 
 
 Let $q = k \leb \Opcl {\wt B}$.  Note that  $\leb \Opcl {\wt C^u} \le q^u$ as before. By the choice of $m^*$ we have  $Z \in \bigcap_{u\ge r} \Opcl {\wt C^u}$, so since $q   < 1$,  an appropriate refinement of the sequence of open sets $\seq {\Opcl {\wt C^u}}\sN u$ shows   $Z$ is not ML-random.

\end{proof}

\section{Towards the general case}
\subsection{Recurrence for $k$   shift operators} 
  The probability space under consideration is  now  $\+ X = \{0,1\}^{\NN^{\normalsize k}}$ with the product measure.
For $1\le i \le k$, the operator $T_i \colon \+ X \to \+ X$ takes one   ``face" of bits off in direction  $i$. That is, for $Z \in \+ X$,  \bc $T_i(Z)(u_1, \ldots, u_k) = Z(u_1, \ldots, u_i+1, \ldots, u_k)$. \ec
$Z$ is \emph{multiply recurrent} in a class $\+ P \sub \+ X$ if $ [Z \in \bigcap_{  i \le k} T^{-n}_i ( \+ P) $ for some $n$.

Algorithmic randomness notions for points in $\+ X$ can be  defined via  the  effective measure preserving isomorphism $\+ X \to \cantor$ given by a computable bijection $\NN^k \to \NN$. Modifying  the methods above,  we show the following. 

\begin{thm} Let $\+ P \sub \+ X$  be a $\PPI$ class with  $ 0< p=\leb  \+ P$.  Let $Z \in \+ X$.

If $Z$ is (a) Kurtz (b) Schnorr (c) ML-random, 
 then $Z$ is  {multiply} recurrent in $\+ P$    

in case (a) $\+ P$ is clopen  (b) $\leb  \+ P$ is computable (c) for any $\+ P$.\end{thm}
\begin{proof} For the duration of this proof, by an \emph{array} we mean a map $\sss \colon \{0, \ldots, n-1\}^k \to \{0,1\}$. We call $n$ the \emph{size} of $\sss$ and write $n= \sssl$. The letters $\sss, \tau, \rho, \eta $ now denote arrays. If $\sss$ is an array of size $n$, for $s \le n$ and $i \le k$ let $(\sss)_{i,s}$ be the array $\tau$ of size $n-s$ such that   \bc $\tau (u_1, \ldots, u_k) = \sss (u_1, \ldots, u_i+s, \ldots, u_k)$ \ec
 for $u_1, \ldots, u_k \le n-s$.  This operation removes $s$ faces in direction $i$, and then in the remaining directions cuts the  faces  from the opposite side in order to obtain an array.
For a  set $S$ of arrays we define $\Opcl S = \{ Y  \in \+ X \colon \, \ex \sss \in S \, [\sss \prec Y]\}$ where the ``prefix"  relation $\prec$ is defined as expected.   

Suppose  that  $Z$ is not $k$-recurrent in $\+ P$ for some $k \ge 1$.

\n (a). As in Prop.\ \ref{prop: Kurtz}  we define a null $\PPI$ class  $\+ Q \sub \+ X$ containing $Z$. Let $n_1$ be least such that $\+ P = \Opcl F$ for some set of arrays of that all have size $n_1$. Let 
\[ \+ Q = \bigcap_{r \ge 1} \{ Y \colon   \bigvee_{1 \le i \le k} T_i^{rn}(Y)  \not \in \+ P\}. \]
 By the choice   of $n_1$  the conditions in  the same disjunction are independent, so  we have \[ \leb (\bigvee_{1 \le i \le k} T_i^{rn}(Y)  \not \in \+ P ) =  1 - p ^k < 1.\] The  $\PPI$ class $\+ Q$ is the independent intersection of  such classes  indexed by~$r$. Therefore $\+ Q$ is null. By hypothesis $Z \in \+ Q$. So $Z$ is not weakly random.
 
 \n (b).  We could modify the previous argument. However, this also follows by  the general fact in Remark~\ref{Rem:Rute} below. 
 
 \n (c). The argument is   similar to the proof of Theorem~\ref{thm:ML} above. The definition of the c.e.\ set $B$ and its enumeration are as before, except that  each string of length $n$ is now an array of size~$n$. In particular, an array enumerated at a stage $s$ has size $s$. 
 
 Let  $C^0$ only contain the empty array,  which is enumerated at stage $0$. Suppose  $r > 0$ and $C^{r-1}$  has been defined. Suppose $\sss$ is enumerated in $C^{r-1}$ at stage $s$ (so $\sssl =s$). 
 
 At a stage $t >  2s$,  for each array  $\eta $ of   size $t$  such that   $\eta \succ \sss$ and 
\[ \tag{$*$} \bigvee_{1 \le i \le k}  (\eta)_{i,s}  \in   B_{t-s}, \]
and no array that is a prefix of $\eta$ is in $C^r_{t-1}$,   put $\eta$  into $C^r$ at stage $t$.  
 As before one checks that   $C^r$ is prefix-free for each $r$.  

Choose a finite set of arrays   $D \sub  B$ such that the set $\wt B = B- D$ satisfies $\leb \Opcl {\wt B} < 1/k$.  
Let $N = \max \{\sssl \colon \sss \in D\}$.  
Let $C = \bigcup_r C^r$. Let $G_m$ be the set of prefix-minimal arrays $\eta  $  such that  $\eta  \in C$, and  there exist $m$ many $s> N$   as follows. 
\bi \item  $\eta \uhr {\{ 0, \ldots, s-1\}^k} \in C$, and \item for some $i$ with  $1 \le i \le k$, $(\eta)_{i,s}$ extends an array in $D$.  \ei 
 The sets $G_m $ are uniformly $\SI 1$. By  choice of $N$ and independence  $\leb \Opcl{G_{m+1}} \le (1- t^k) \leb G_m$, where $t = \leb (\cantor - \Opcl D)$. If  $Z$ is ML-random we can choose a least $m^*$ be such that $Z \not \in \Opcl{G_{m^*}}$, and  $m^*>0$ since $G_0 = \{\ES\}$. So choose $\eta \prec Z$ such that $\eta \in G_{m^*-1}$. Then 
 $\eta \in C^r$ for some  $r$, and  no $\tau$ with  $\eta \preceq \tau \prec Z$ is in $G_{m^*}$. 
 

  Let $\wt C^r = C^r$. Suppose  $u >  r$ and $\wt C^{u-1}$ has been defined. For each  $\sss \in \wt C^{u-1}$,      put into   $\wt C^{u} $ all the  arrays  $\eta \succ \sss$ in $C^{u} $ so that    ($*$) can be strengthened to   $\bigvee_{1 \le i \le k}  (\eta)_{i,s}  \in   \wt B_{t-s}$, where $s = \sssl$ and $t = |\eta|$. 
 
 Let $q = k \leb \Opcl {\wt B}$.  Then  $\leb \Opcl {\wt C^u} \le q^u$ as before. By the choice of $m^*$ we have  $Z \in \bigcap_{u\ge r} \Opcl {\wt C^u}$, so since $q   < 1$,    $Z$ is not ML-random.
\end{proof} 

\subsection{The  putative full result} It is likely that a multiple recurrence theorem holds in greater generality. For background on computable  probability spaces and how to define randomness notions for points in them, see e.g.\ \cite{Gacs.Hoyrup:11}.


\begin{conjecture} \label{big conj}Let $(X, \mu) $ be a computable probability space. Let $T_1, \ldots, T_k$ be computable measure preserving transformations that commute pairwise. Let $\+ P $ be a $\PPI$ class with $\mu P> 0$. 

If $z \in \+ P$ is ML-random then $\ex n  [ z \in \bigcap_{  i \le k} T^{-n}_i ( \+ P)]$. 
\end{conjecture}

\begin{remark} \label{Rem:Rute} {\rm   Let $U_n $ be the open set $ \{x \colon \,   x \not \in \bigcap_{i\le k} T^{-n}_i ( \+ P)\}$. Then $\mu ( \+ P \cap \bigcap_n U_n)=0$ by the classic multiple recurrence  theorem in the version of Cor.\ \ref{prop:FurMRT3}. Since $\+ P \cap \bigcap_n U_n$  is $\PI 2$, weak 2-randomness of $z$ suffices for the $k$-recurrence.

  Jason Rute has pointed out that if $X$ is Cantor space and $\mu \+ P$ is computable, then $\ex n [ z \in \bigcap_{  i \le k} T^{-n}_i ( \+ P)]$  for every Schnorr random $z \in \+ P$.  For in this case $\mu  \widehat U_n$ is uniformly computable where $\widehat U_n = \bigcap_{i< n}U_i$. Let $\+ P = \bigcap_n \+ P_n$ where the $\+ P_n$ are clopen sets computed uniformly in $n$. Let $G_n = \+ P_n \cap \widehat U_n$. Then $G_n$ is uniformly $\SI 1$ and $\mu (G_n)$ is   uniformly computable. Refining the sequence $\seq{G_n}$ we obtain a Schnorr test capturing $z$. }  \end{remark}


\section{Effective Multiple Recurrence for Irrational Rotations}

In the foregoing sections, we have seen examples of   mixing systems that  display
effective versions of Furstenberg multiple recurrence. At the other
end of the spectrum, highly structured systems also exhibit multiple
recurrence.  Such systems are called compact. An example is irrational rotations of the unit circle. rotations. 
We discuss   a  fact from  \cite{Furstenberg:2014} which implies that \emph{every}
point in such a system is multiply recurrent with respect to every of its neighbourhoods, rather than merely the
Kurtz random points.  
To establish this result, we briefly look at  topological dynamical systems. Such a  system
consists of a compact space $X$ and a continuous operator $T \colon X
\to X$. 

\begin{definition}[\cite{Furstenberg:2014}, Def.\ 1.1] 
A point $x$ in a topological dynamical system $(X,T)$ is \emph{$k$-recurrent}
if the condition of Definition~\ref{def:central} holds for each
neighbourhood $\+ P$ of $x$. \end{definition}

\begin{definition}
Given a  compact group $G$ and $a \in G$,  let  $T_a(x) = a \cdot x$. One  calls   $(G,T_a)$ a
\emph{Kronecker system}.
\end{definition}

For instance, for  $\alpha \in \RR$, the system  $(\mathbb{R}/\mathbb{Z}, T)$  where $T(x) = \alpha + x \mod 1$ is a Kronecker system. There is a unique invariant probability measure, called the Haar measure, on a Kronecker system. Hence  such a system  can also be viewed as  a measure-preserving system, which turns out to be  compact in the measure theoretic sense. It is known that every compact \emph{ergodic} system is equivalent to a Kronecker system in the sense that the two  systems are isomorphic when viewed on the $\sigma$-algebra of  measurable sets mod null sets.

\begin{lemma}(see \cite{Furstenberg:2014}, Chapter 1)
Every point in a Kronecker system is $1$-recurrent.
\end{lemma}

This is established by showing that there is some recurrent point $x_0
\in G$, by first considering a minimal subsystem consisting of points
with dense orbits, and by applying the Zorn's lemma. Since $G$ is a
group,  if there  is any recurrent point in
the system, then every point must be recurrent. 
 
 We can use the lemma itself to strengthen it. 
\begin{lemma}
Every point in a Kronecker system $(G,T_a)$ is multiply recurrent.
\end{lemma}
\begin{proof}
Given an integer $k \ge 2$, consider the tuple $t= (a, a^2, \ldots, a^k)$ in the compact group $H=G^k$. The system $(H, t)$ is also   Kronecker. Every point is $1$-recurrent in this system. In particular, for every $x \in G$, the point $y= (x, \ldots, x) \in H$ is $1$-recurrent. 

If $V$ is an open nbhd of $x$, then the cartesian power ${}^kV$ is an open nbhd of $y$. So $y \cdot  t^n \in {}^kV$ for some $n$. This means that $x a^{in} \in V$ for each $i \le k$, as required. 
\end{proof}  

Let $\alpha $ be a computable irrational. Brown Westrick (see \cite{LogicBlog:16}) has proved that every  ML-random point  in $[0,1]$ is multiply recurrent  in each  $\PPI$ class of positive measure  for at least one of the operators $x \to (x+ \alpha) \mod 1$ or $x \to (x- \alpha) \mod 1$. 
Note that this lends some further evidence to Conjecture~\ref{big conj}.  However, even for    Kronecker systems with computable group structure and computable Haar measure, the conjecture is   open.
%

\def\cprime{$'$} \def\cprime{$'$}

\end{document}